\documentclass[10pt]{amsart}
\usepackage{amsmath}

\usepackage{graphicx}
\usepackage{hyperref}

\usepackage[sort]{natbib}
\usepackage{setspace}
\usepackage{amsthm}
\usepackage{amssymb}
\usepackage{amsfonts}
\usepackage{mathpazo}
\usepackage{xcolor}
\usepackage[T1]{fontenc}
\usepackage[utf8]{inputenc}

\def \N {{\mathbb N}}
\def \R {{\mathbb R}}
\def \H {{\mathbb H}}
\def \C {{\mathbb C}}
\def \M {{\mathbb M}}

\def \D {{\mathcal D}}
\def \Ul {{\mathcal U}}
\def \Mc {{\mathcal M}}
\def \W {{\mathbb W}}
\def \V {{\mathcal V}}
\def \G {{\mathcal{G}}}

\def \tr {{\operatorname{tr}}}
\def \Re {{\operatorname{Re}}}

\newtheorem{theorem}{Theorem}[section]

\newtheorem{corollary}[theorem]{Corollary}

\newtheorem{remark}[theorem]{Remark}
\newtheorem{pro}[theorem]{Proposition}

\newtheorem{definition}[theorem]{Definition}
\newtheorem{lemma}[theorem]{Lemma}

\title[Properties of best approximations with respect to the Ky Fan $p$-$k$ norm, and the strict spectral approximant of a matrix]{Properties of best approximations with respect to the Ky Fan $p$-$k$ norm, and the strict spectral approximant of a matrix}
\author[ P. Grover, \ K. K. Gupta ]
{Priyanka Grover$^{1}$, Krishna Kumar Gupta$^{2}$ }
\address{{$^{1}$}   Priyanka Grover,
	Department of Mathematics, 
	Shiv Nadar Institution of Eminence Delhi NCR, NH-91, Tehsil Dadri, Uttar Pradesh 201314, India 
}
\email{priyanka.grover@snu.edu.in}

\address{{$^{2}$}   Krishna Kumar Gupta,
	Department of Mathematics, 
	Shiv Nadar Institution of Eminence Delhi NCR, NH-91, Tehsil Dadri, Uttar Pradesh 201314, India 
}
\email{kg952@snu.edu.in, shrikrishna6996@gmail.com}

\subjclass{15A60, 15A18, 58C20, 47A30}
\keywords{Ky Fan $p$-$k$ norms, Subdifferential of norm, Best approximations of matrices, Approximate orthogonality, P\'olya algorithm, Strict Spectral approximant}

\begin{document}

\begin{abstract}
	Some questions raised in [K. Zi\k{e}tak, {\it From the strict Chebyshev approximant of a vector to the strict spectral approximant of a matrix}, Warsaw :  Banach Center Publ., 112  Polish Acad. Sci. Inst. Math.  (2017)] are discussed. To do so, the subdifferential set of the Ky Fan $p$-$k$ norm is computed.
A characterization for the best approximations with respect to the Ky Fan $p$-$k$ norms is given. Further, necessary and sufficient conditions for  $\varepsilon$-Birkhoff orthogonality with respect to the Ky Fan $p$-$k$ norm are also derived.     
\end{abstract}
\maketitle
	\section{Introduction}
  Let $\mathbb{M}_{m\times n}(\C)$ denote the set of $m\times n$ matrices over $\C$ with a given norm $\|\cdot\|.$ Let $\mathcal{M}$ be a subspace of $\M_{m\times n}(\C).$ Let $A\in \M_{m\times n}(\C).$ A matrix $Y\in\Mc$ is called a best approximation of $A$ with respect to $\|\cdot\|$  if
  \begin{equation}\label{kp_approx_def11}
  	\|A-Y\|=\min_{X\in \Mc}\|A-X\|.
  \end{equation} 
  When $Y$ is the zero matrix, we have $\|A\|\leq \|A-X\|$ for all $X\in \Mc$. In this case, it is said that $A$ is \emph{Birkhoff-James orthogonal} to $\Mc$.  If $A$ is a Hermitian matrix and $\Mc$ is the real subspace of diagonal matrices, then $A$ is called a \emph{minimial matrix}. See \cite{andruchow} for applications of minimal matrices to flag manifolds of lower dimensions. More recently, \cite{ying} deals with minimal Hermitian matrices related to certain $C^*$-subalgebras of $\mathbb M_{n \times n}(\mathbb C)$. In \cite{PGmatrixsubspace}, some of these problems have been studied in much more detail by considering Birkhoff-James orthogonality to any real or complex subspace (not just $C^*$-subalgebras) of $\mathbb M_{n \times n}(\mathbb C)$. In \cite{surveyZ}, an alternate proof was given for complex subspaces of $\M_{m\times n}(\C)$. More related results with respect to the various norms can be found in \cite{PriGro_thesis}.
 Let $n_0:=\min\{m,n\},$ and write
  $\boldsymbol{\sigma}(A)=\big(\sigma_1(A),\ldots,\sigma_{n_0}(A)\big)\in\R^{n_0}$ for the vector of singular values of $A,$ listed in nonincreasing order.
 For $1\leq p<\infty,$  the Schatten $p$-norm or the $c_p$ norm of $A$ is 
 \begin{equation}\label{cp_defi}
 	\|A\|_p=\left(\sum\limits_{i=1}^{n_0}\sigma_i^p(A) \right)^\frac{1}{p}.
 	\end{equation}

For $p=1,$ it is called the trace norm of $A.$ The spectral norm of $A$ is 
$\|A\|_\infty=\sigma_1(A).$

 For $1<p<\infty,$ the $c_p$ norm is strictly convex \cite{zietakextremal}. So the best approximation of $A$ in $c_p$ norm is unique. We call this the $c_p$-approximation of $A$ and denote it by $Y_p.$ Since the spectral norm and the trace norm are not strictly convex, spectral approximation and trace approximation may not be unique. To address this, the notion of {\it canonical trace class approximation} was introduced in \cite{Legg_wardtrace_1985}, where it was also shown that $c_p$-approximation of $A$ converges to the canonical trace class approximation of $A$ as $p \to 1.$ Similarly, the concept of {\it strict spectral approximation} was introduced in \cite{SSA}. A matrix $Y^{(st)}\in\Mc$ is called the strict spectral approximation if the vector $\boldsymbol{\sigma}(A-Y^{(st)})$ is minimal with respect to the lexicographic ordering on the set $\{\boldsymbol{\sigma}(A-Y):Y\in\Mc\}.$ Another definition of strict spectral approximation is described in \cite{SSA} as follows. Let $1\leq p<\infty$ and $1\leq k\leq n_0.$ Then, the Ky Fan $p$-$k$ norm of $A$ is defined as:
 \begin{equation}
 	\|A\|_{(p,k)}=\left(\sigma^p_1(A)+\sigma^p_2(A)+\cdots+\sigma^p_k(A)\right) ^\frac{1}{p}. 
 \end{equation} The spectral norm $\|A\|_\infty$ is $\|A\|_{(p,1)}$. For $1 \leq k \leq n_0$ and $p=1$, 
 $\|A\|_{(1,k)} = \sigma_1(A) + \cdots + \sigma_k(A)$ defines the Ky Fan $k$ norm. Let $\Mc_1$ denote the set of spectral approximations of $A$ in $\Mc.$ For $p=2,$ the nested sequence of sets $\Mc_k$ is defined for each $k=1,\ldots,n_0$ by
 $\Mc_k=\{Y\in \Mc_{k-1}:\|A-Y\|_{(2,k)}=\min\limits_{X\in\Mc_{k-1}}\|A-X\|_{(2,k)}\}.$ Thus we get \begin{equation}\label{subset_seq}
 	\Mc\supset\Mc_1\supset\Mc_2\ldots\supset \Mc_{n_0-1}\supset\Mc_{n_0}.
 	\end{equation} Then, the matrix contained in $\Mc_{n_0}$ is the strict spectral approximation $Y^{(st)}.$
 
 In \cite{surveyZ}, it was conjectured that \begin{equation}\label{conjecture}
 Y_p\to Y^{(st)} \quad \text{ as } p\to\infty.
 \end{equation}  
  In the same paper, an attempt to prove the conjecture was made.  However, due to limited information about the best approximations in the Ky Fan $p$-$k$ norm, the proof remains incomplete. Some questions related to the convergence relations about the best approximations with respect to the Ky Fan $p$-$k$ norm were raised in \cite{surveyZ} and it was suggested to find the \emph{subdifferential} of the Ky Fan $p$-$k$ norm. We answer the latter and then explore some properties of best approximations with respect to $\|\cdot\|_{(p,k)}$. The proof of \eqref{conjecture} is given in special cases. 
  
  Let $(\mathcal{X},\|\cdot\|)$ be any normed space and let $f:\mathcal{X}\to\R$ be a continuous convex function. For $a\in\mathcal{X},$ the subdifferential set of $f$ at $a$ is defined as 
  \begin{equation}
  	\partial f(A)=\{\psi\in\mathcal{X}^*:\Re\, \psi(x-a)\leq f(x)-f(a) \text{ for all } x\in \mathcal{X}\}.
  \end{equation}
  For any norm $\|\cdot\|$ on the space $\M_{m\times n}(\C)$, it is well known that
  \begin{equation}
  	\partial \|A\|=\{G\in \M_{m\times n}(\C): \|A\|=\text{Re } \tr(G^* A),\|G\|^*\leq 1\},\label{eq 1.4.1}
  \end{equation}
  where $\|\cdot\|^*$ is the dual norm of $\|\cdot\|$. For a unitarily invariant norm (one that satisfies $|\!|\!|\Ul A \V|\!|\!| = |\!|\!|A|\!|\!|$ for all $A\in\M_{m\times n}(\C) $ and all unitary matrices $\Ul\in\M_{m\times m}(\C)$, $\V\in\M_{n\times n}(\C)$ ), there exists a unique symmetric gauge function $\psi$ on $\mathbb{R}^{n_0}$ such that
  $
  |\!|\!|A|\!|\!| = \psi\big((\sigma_1(A), \ldots, \sigma_{n_0}(A))\big) \text{ for every } A \in \M_{m\times n}(\C).$ 
  In \cite{watson_subdiff_cha}, for the space of real $m\times n$ matrices, an expression of $\partial|\!|\!|\cdot|\!|\!|$ was given in terms of the respective symmetric gauge function. It was shown in \cite{Zietaksubdual1993} that for $A \in \mathbb{M}_{m\times n}(\mathbb{C})$,
  \begin{equation}\label{subd_ziet_gauge1}
  	\begin{aligned}
  	\partial |\!|\!|A|\!|\!| 
  	 & = \Big\{\Ul D \V^* : A = \Ul \operatorname{diag}(\sigma_1(A),\ldots, \sigma_{n_0}(A)) \V^* \text{ is a singular value}\\
  	 & \text{ decomposition of } A,D=\operatorname{diag} 
  	 (d_1, \ldots, d_{n_0})\in\M_{m\times n}(\C),
  	\sum \sigma_i(A)d_i = |\!|\!|A|\!|\!| = \\
  	&\psi(\sigma_1(A), \ldots, \sigma_{n_0}(A)), \psi^*(d_1, \ldots, d_{n_0}) = 1 \Big\}.
  \end{aligned}
  \end{equation}
   In Theorem \ref{subdif_kyFan KP}, we give more explicit expressions for the subdifferential sets of the Ky Fan $p$-$k$ norms for $p\geq 2$. The subdifferential sets have been used in the study of best approximation \cite{singer2013best,tr,Watson1993KyFank,Prop_spectrZietak1993}. For further details, refer to \cite{ PGmatrixsubspace, Bhatt_Priyanka,Watson1993KyFank,KyFan_ortho_Grover}.)
  Recently, it has been used to investigate minimal self-adjoint operators for matrices and compact operators in \cite{BOTtaziMinimal_com}.
  
  Let $Y_\infty\in\Mc$ be a spectral approximation of $A.$ Let $R_\infty=A-Y_\infty$. Let $s_0$ be the multiplicity of $\sigma_1(R_\infty)$ and let $I_{s_0}$ denote the $s_0\times s_0$ identity matrix. Let $R_\infty=W\operatorname{diag}(\boldsymbol{\sigma}(R_\infty))Z^T=\sigma_1(R_\infty)W_1Z_1^T+W_2\Sigma_2Z_2^T
  $
  be a singular value decomposition of $R_\infty$, where $W=[W_1,W_2]$, $\operatorname{diag}(\boldsymbol{\sigma}(R_\infty))=\Sigma_1\oplus\Sigma_2$, 
  and $Z=[Z_1,Z_2]$ with $W_1,Z_1\in\M_{m\times s_0}$ and $\Sigma_1=\sigma_1(R_\infty)I_{s_0}$. It was shown in \cite[Theorem 7.3]{surveyZ} that for every spectral approximation $\widehat{Y}\in\Mc$, there exists $\widehat{\Sigma_Y}$ with $\|\widehat{\Sigma_Y}\|_\infty\leq\sigma_1(R_\infty)$ such that
  $
  \widehat{Y}=A-\sigma_1(R_\infty)W_1Z_1^T-W_2\widehat{\Sigma_Y}Z_2^T.$
 For more similar properties of spectral approximations, see \cite{Prop_spectrZietak1993,Liesen_Tichy2009,PGmatrixsubspace}.
   In~\cite{Watson1993KyFank, KyFan_ortho_Grover}, best approximations with respect to the Ky Fan $k$ norms are characterized. In Theorem \ref{kthsingulavaleRpk1}, we present a characterization of the best approximations with respect to the Ky Fan $p$- $k$ norms. A corresponding characterization in $\mathbb{R}^n$ with respect to the $ k$-major $ \ell_p$ norm (the restriction of the Ky Fan $ p $-$ k$ norm to the space of  $n \times n$ real diagonal matrices) is presented in \cite{watson1linear9k9major4lp}.
  
  In Section~\ref{sec2}, we present an explicit expression for the subdifferential set of the Ky Fan $p$-$k$ norm for $p\geq 2.$ As a consequence, we give a characterization of approximate orthogonality and characterize the Birkhoff-James orthogonality to a subspace of $\M_{m\times n}(\C)$ with respect to the Ky Fan $p$-$k$ norm. In Section \ref{sec3}, we give some results towards understanding the properties of best approximations with respect to the Ky Fan $p$-$k$ norm and give the proof of \eqref{conjecture} in special cases.   
  \section{Subdifferential of the Ky Fan $p$-$k$ Norm}\label{sec2}
  We first observe that, by taking $(A^*A)^\frac{p}{2}$ as the Hermitian matrix in \cite[Theorem A.2., Ch.~20]{marshall1979inequalities}, another expression of the Ky Fan $p$-$k$ norm of $A$ can be given as follows:
  \begin{equation}\label{1}
  	\begin{aligned}
  		\|A\|_{(p,k)}
  		&= \left(\max _{\substack{U\in\M_{n\times k}(\C) \\ U^* U=I_k}} \Re\, \tr( U^* (A^*A)^\frac{p}{2} U)\right)^\frac{1}{p}\\
  		&= \max _{\substack{U\in\M_{n\times k}(\C) \\ U^* U=I_k}}\left( \Re\, \tr( U^* (A^*A)^\frac{p}{2} U)\right)^\frac{1}{p}.
  	\end{aligned}
  \end{equation}

  We now recall a result on the Fr\'echet differentiability of matrix-valued functions defined over a subset of Hermitian matrices.
    Before stating the result, we introduce the following notations.
    
    Let $\mathcal{I}\subset\R$ be the open interval. Let $C^1(\mathcal{I})$ denote the space of continuously differentiable real-valued functions on $\mathcal{I}.$ For a function $f\in C^1(\mathcal{I}),$ define the map $f^{(1)} : \mathcal{I} \times \mathcal{I} \to \mathbb{R}$ by
    $$
    f^{(1)}(\lambda, \alpha) = 
    \begin{cases}
    	\displaystyle\frac{f(\lambda) - f(\alpha)}{\lambda - \alpha}, & \lambda \neq \alpha, \\
    	f'(\lambda), & \lambda = \alpha.
    \end{cases}
    $$ Let $\H_n(\C)$ denote the set of all $n \times n$ Hermitian matrices. Let $\G(\mathcal{I})$ be the set of Hermitian matrices whose eigenvalues lie in $\mathcal{I}.$ Let $B\in\mathcal{G}(\mathcal{I})$ with spectral decomposition $B = UDU^*,$ where $D=\operatorname{diag}(\lambda_1,\dots,\lambda_{n}).$ Define the induced map $f:\G(\mathcal{I})\to\H_n(\C)$ by $f(B)=Uf(D)U^*,$ where $f(D)=\operatorname{diag}(f(\lambda_1),\dots,f(\lambda_{n})).$ Let $\mathcal{D}f(B)$ denote the Fr\'echet derivative of the function $f$ at $B$.
     \begin{pro}\cite{bhatia2013matrix, bhatia2009positive}\label{bhatiadiff}
    	Let $B\in\G(\mathcal{I})$ be as above. Then, for any $H\in\H_n(\C),$
    	$$\mathcal{D}f(B)(H) = U\left[f^{(1)}(D) \circ (U^*HU)\right]U^*,$$
    	where $f^{(1)}(D)$ is the matrix whose $(i,j)$-entry is $f^{(1)}(\lambda_i, \lambda_j),$ and $\circ$ denotes the Schur-product of two matrices.
    \end{pro}
   This is Equation~(V.13) in \cite{bhatia2013matrix}, and its proof is presented in \cite[Theorem 5.3.3]{bhatia2013matrix} for the interval $\mathcal{I}=(-1,1).$ With minor modifications, the same proof can be extended to any open interval $\mathcal{I}.$ Alternatively, another proof can be found in \cite[Theorem 5.3.1]{bhatia2009positive}.
    
  For $p> 2,$ let $\hat{f} : (-1, \infty) \to \R $ be defined by $\hat{f}(x) = |x|^{\frac{p}{2}}.$ Then $\hat{f} \in C^1((-1, \infty))$ and $\hat{f}$ induces a map on $\G((-1, \infty)).$ Using the above proposition, we get the following lemma. 
  
 \begin{lemma}\label{fv_derivative}
  	 Let $A\in\M_{m\times n}(\C)$ and $ 2<p<\infty .$ Let $V$ be a $n\times k$ matrix whose columns are $v_1,v_2,\ldots,v_k$ satisfying $A^*Av_i=\sigma_i^2(A)v_i,$ for $1\leq i\leq k.$ Let $f_V:\M_{m\times n}(\C)\to \R$ be a function defined as  \begin{equation*}
  		f_V(X)=\Re\,\tr(VV^*(X^*X)^\frac{p}{2})=\Re\,\tr\left(\sum\limits_{i=1}^{k}v_iv_i^*(X^*X)^\frac{p}{2}\right) \text{ for } X\in\M_{m\times n}(\C).
  	\end{equation*}  Then, the derivative of $f_V$ at $A$ is given by \begin{equation}\label{derivativefv}
  		\D f_V(A)(X)=p\,\Re\,\tr\left(\big(AVV^*(A^*A)^\frac{p-2}{2}\big)^*X\right) \text{ for } X\in\M_{m\times n}(\C).
  	\end{equation}
  \end{lemma}		 
  \begin{proof} For each $V\in \M_{n\times k}(\C)$, let $h_V:\H_{n}(\C)\to \R$ be the differentiable function defined as $h_V(H)=\Re\,\tr(VV^*H)$ for $H\in\H_{n}(\C).$ Let $g_1:\M_{m\times n}(\C)\to \G$ be the differentiable function defined as $g_1(Y)=Y^*Y$ for $Y\in\M_{m\times n}(\C).$  
  	Thus, for $A\in \M_{m\times n}(\C),$ $f_V(A)=(h_V\circ \hat{f}\circ g_1)(A)$. By the chain rule, we get \begin{equation}\label{chainrule}
  		\D f_V(A)(X)=\D h_V\big(\hat{f}\circ g_1(A)\big)\Big(\D \hat{f}(g_1(A))\big(\D g_1(A)(X)\big)\Big).
  	\end{equation} 
  	Since $A^*A\in\G((-1, \infty)),$ by Proposition \ref{bhatiadiff} we get
  	$$\D \hat{f}(g_1(A))\big(\D g_1(A)(X)\big)=\D \hat{f}(A^*A)(A^*X+X^*A)=V[\hat{f}^{(1)}(\Sigma^*\Sigma)\circ(V^*(A^*X+X^*A)V)]V^*,$$ where  $\Sigma^*\Sigma=\operatorname{diag}[\sigma_1^2(A),\ldots,\sigma_n^2(A)].$ 
  	Observe that the $(i,j)$th entries of $(XV)^*A
  	V$ and $(AV)^*XV$ are $\langle X v_i, Av_j\rangle$ 
  	 and  $\langle Av_i, Xv_j\rangle,$ 
  	  respectively. Additionally,  $V^* \sum_{i=1}^{k}v_iv_i^*V=\begin{bmatrix}
  		I_{k\times k}&0\\0&0
  	\end{bmatrix}.$ Thus,
  	 \begin{align*}\D f_V(A)(X)=&\Re\,\tr\left(\sum_{i=1}^{k}v_iv_i^*\left(\D \hat{f}(A^*A)(A^*X+X^*A)\right)\right)\\
  		&=p\sum_{i=1}^{k}\sigma_i^{p-2}(A)\Re\langle Av_i,Xv_i \rangle\\
  		&=p\sum_{i=1}^{k}\Re\langle X^*Av_i,(A^*A)^\frac{p-2}{2}v_i \rangle\\
  		&=p\,\Re\,\tr \left(\big(AVV^*(A^*A)^\frac{p-2}{2}\big)^*X\right).   
  	\end{align*}

  \end{proof}
  The following proposition follows from \cite[Theorem 2.4.18]{zalinescuGeneralconvex}.
  \begin{pro}[\cite{zalinescuGeneralconvex}]\label{supconsubdif}
  	Let J be a compact set in some metric space. Let $\{f_j\}_{j\in J}$ be a collection of continuous convex functions from $\M_{m\times n}(\C)$ to $\R$ such that for each $M\in\M_{m\times n}(\C),$ the maps $j \to f_j(M)$ are upper semi-continuous. Let $g : \M_{m\times n}(\C) \to \R$ be defined as $g(M) = \sup\limits_{j \in J}f_j(M).$ For  $N\in\M_{m\times n}(\C),$ let $J(N) = \{i \in J:g(N)=f_i(N)\}.$ Then  $$\partial g(N) = \operatorname{conv}(\cup\{\partial f_i(N) :i \in J(N)\}).$$
  \end{pro} 
  We now state the main result of this section.
  \begin{theorem}\label{subdif_kyFan KP}
  	Let $A\in \M_{m\times n}(\C)\setminus\{0\}$ and $2\leq p< \infty$. Then
  	\begin{equation}
  		\begin{split}\label{subdiff}
  			\partial\|A\|_{(p,k)}=\operatorname{conv}&\Bigg\{\frac{1}{\|A\|_{(p,k)}^{p-1}}A(A^*A)^\frac{p-2}{2}\sum_{i=1}^{k}v_iv_i^*: v_1,v_2,\ldots,v_k \text{ o.n. vectors with }\\
  			& A^*A v_i=\sigma_i^2(A) v_i \text{ for all }1\leq i \leq k \Bigg\}.
  		\end{split}
  	\end{equation}
  	
  \end{theorem}
  
  \begin{proof}
  	Let $J=\{U\in \M_{n\times k}(\C)$: $U^*U=I_k\}$. It is a compact subset of $\M_{n\times k}(\C)$. 
  
  		\textit{Case 1}. $p=2$.  For $U\in J,$ let $h_U:\H_{n}(\C)\to \R$ be the differentiable function defined as $h_U(H)=\Re\,\tr(UU^*H).$ Let $g_1:\M_{m\times n}(\C)\to \G$ be the differentiable function defined as $g_1(B)=B^*B.$
  		By the chain rule, we have for any $X\in \M_{m\times n}(\C),$ 
  		\begin{align*}
  			\D (h_U\circ g_1)(A)(X)&=\D h_U( g_1(A))\Big(\D g_1(A)(X)\Big)\\
  			&=\Re\,\tr (UU^*(A^*X+X^*A))\\
  			&=2\ \Re\,\tr(UU^*A^*X).
  		\end{align*} 
  		Now $\|A\|_{(2,k)}=\max\limits_{U\in J}\Big(h_U\circ g_1(A)\Big)^\frac{1}{2}.$ Thus, by Proposition \ref{supconsubdif}, we get the required result.
  		
  		\textit{Case 2}.
  	For $U\in J,$ let $\phi_U:\M_n(\C)\to  \R$ be  defined as $\phi_U(B)=(\Re\, \tr( UU^* (B^*B)^\frac{p}{2}))^\frac{1}{p} .$ Then, we get $$\|A\|_{(p,k)} = \max _{U\in J}\phi_U(A).$$ Let $J(A)=\{V\in J:\|A\|_{(p,k)}=\phi_V(A)\}.$ The set $J(A)$ will be equal to the set of $n\times k$ matrices whose columns are  right singular vectors of $A,$ that is, $v_1,v_2,\ldots,v_k$ corresponding to singular values $\sigma_1(A),\sigma_2(A),\ldots,\sigma_k(A),$ respectively. Let $f_V$ be the map as in Lemma \ref{fv_derivative}. This will imply that for $V\in J(A),$ $\phi_V=f_V^\frac{1}{p}.$ 
  	Since $f_V$ is a differentiable map, $\phi_V$ is also differentiable. By applying the chain rule to $\phi_V$, we obtain 
  	$$\D \phi_V(A)(X)=\frac{1}{p\|A\|_{(p,k)}^{p-1}} \D f_V(A)(X).$$ By Lemma \ref{fv_derivative}, $ \partial \phi_V(A)=\{\D\phi_V(A)\}=\left\{\frac{1}{\|A\|_{(p,k)}^{p-1}}AVV^*(A^*A)^\frac{p-2}{2}\right\}.$ Thus, by using Proposition \ref{supconsubdif}, we get \eqref{subdiff}.
  \end{proof}

  \begin{remark}\label{remark_ext}	
  	\begin{enumerate}
  		\item	Let $2\leq p<\infty.$ Let $A\in\M_{m\times n}(\C)$ such that $\sigma_k(A)>\sigma_{k+1}(A).$ Then, by \cite[Theorem 2.6.5]{horn2012matrix}, the set $\partial\|A\|_{(p,k)}$ is a singleton.
  		\item For $2\leq p<\infty,$ let \begin{align*}
  			\mathcal{K}=\Bigg\{\frac{1}{\|A\|_{(p,k)}^{p-1}}A(A^*A)^\frac{p-2}{2}\sum_{i=1}^{k}v_iv_i^*:& v_1,v_2,\ldots,v_k\text{ o.n vectors satisfying }\\ 
  			&
  			A^*A v_i=\sigma_i^2(A) v_i \text{ for all }1\leq i \leq k \Bigg\}.
  		\end{align*}  Let $q=\frac{p}{p-1}.$ Then $\|K\|_q=1 \text{ for all } K\in\mathcal{K}.$ Since the $c_q$ norm is strictly convex, it follows that $\mathcal{K}$ is the set of extreme points of $\partial\|A\|_{(p,k)}.$
  	\end{enumerate}
  	
  \end{remark}
  Let $A,X\in \M_{m\times n}(\C).$ Then the right hand directional derivative of $\|\cdot\|_{(p,k)}$ at $A$ in the direction of $X$ is defined as $$\D^+\|A\|_{(p,k)}(X)=\lim_{t\to 0^+}\frac{\|A+tX\|_{(p,k)}-\|A\|_{(p,k)}}{t}.$$ 
  \begin{corollary}
  	Let $A,X\in \M_{m\times n}(\C).$ Let $2\leq p<\infty$ and $1\leq k\leq n_0.$ Then
  	$$\D^+\|A\|_{(p,k)}(X)=\max_{\substack{v_1,\ldots,v_k \text{o.n.}\\A^*Av_i=\sigma_i^2(A)v_i}}\frac{1}{\|A\|_{(p,k)}^{p-1}}\sum_{i=1}^{k}\sigma_i^{p-2}(A)\,\Re \langle
  	Av_i,Xv_i\rangle.$$	 
  	
  \end{corollary}
  \begin{proof}
  	By \cite[Theorem 1.2.9]{PriGro_thesis}, we have
  	$$\D^+\|A\|_{(p,k)}(X)=\max_{F\in \partial\|A\|_{(p,k)}}\Re\, \tr(X^*F)=\max_{F\in \mathcal{K} }\Re\, \tr(X^*F),$$ where $	\mathcal{K}$ is the set defined in Remark \ref{remark_ext}. This gives the required result. 
  \end{proof}
  
   The right hand derivative is directly related to the concept of approximate Birkhoff orthogonality, see \cite[Theorem 2.2]{ACCF}. The concept of \emph{Birkhoff-James orthogonality} \cite{Birkhoff1935,RCJames_ortho} plays a fundamental role in the study of the geometry of Banach spaces.
 For $A,B\in\M_{m\times n}(\C)$, $A$ is said to be Birkhoff-James orthogonal to $B$ if $\|A+\lambda B\|\geq \|A\|$ $\text{ for all } \lambda \in \C.$ A matrix $A$ is said to be Birkhoff-James orthogonal to a subspace $\mathcal{M}$ of $\M_{m\times n}(\C)$ if $\|A+B\|\geq \|A\|$ $\text{ for all } B\in \Mc.$ Subsequently, certain approximation concepts extending notion of this orthogonality were introduced in \cite{epsilon_app,onapproxSSDrag}. In \cite{epsilon_app}, Chmieli\'nski introduced the notion of $\varepsilon$-Birkhoff orthogonality. This was later characterized in the space of bounded linear operators defined on normed spaces in \cite{ArDeKpOnsomegeometric}.
  \begin{definition}
  	\cite{epsilon_app}
  	Let $A,B\in\M_{m\times n}(\C)$ and $\varepsilon \in [0,1)$. Then $A$ is said to be $\varepsilon$-Birkhoff orthogonal to $B$ if $\|A+\lambda B\|^2\geq \|A\|^2-2\varepsilon\|A\|\|\lambda B\| \, \text{ for all }  \lambda \in \C.$ We denote this relation as  $A\perp_{\|\cdot\|}^\varepsilon B.$
  \end{definition} In the subsequent discussion, using the description of $\partial \|\cdot\|_{(p,k)}$, we characterize $\varepsilon$-Birkhoff orthogonality as well as give a necessary condition for orthogonality to a subspace 
  with respect to the Ky Fan $p$-$k$ norm. 
   To do so, we recall the following well known results from subdifferential calculus.  \begin{pro}\cite{hiriartREalfundamentals,PriGro_thesis}\label{basic_minma_subd} Let $\mathcal{X}$ be a normed space. A continuous convex function $f:\mathcal{X}\to \R$ attains its minima at $a\in\mathcal{X}$ if and only if $0\in \partial f(a).$\end{pro} \begin{pro}\cite{hiriartREalfundamentals,PriGro_thesis}\label{subdif_affine}	Let $\mathcal{X}$ and $\mathcal{Y}$ be Banach spaces. Let $S:\mathcal{X} \to \mathcal{Y}$ be a bounded linear map, and let $L:\mathcal{X} \to \mathcal{Y}$ be the affine map defined as $L(x) = S(x) + y_0$, where $y_0 \in \mathcal{Y}$. If $g:\mathcal{Y} \to \mathbb{R}$ is a continuous convex function, then 	$$\partial(g \circ L)(a) = S^* \partial g(L(a)) \,\, \text{for all } a \in \mathcal{X},$$ 	where $S^*$ denotes the adjoint of $S$, either real or complex, depending on whether $\mathcal{X}$ and $\mathcal{Y}$ are real or complex Banach spaces.\end{pro}
  \begin{pro}\cite{hiriartREalfundamentals,PriGro_thesis}\label{sum}
  	Let $f_1, f_2$ be two continuous convex functions from a Banach space $X$ to another Banach space $Y$. Then 
  	$$\partial(f_1+f_2)(x)=\partial f_1(x)+\partial f_2(x)\,\, \text{ for all }x\in\mathcal{X}.$$
  \end{pro}
  A proof of the next corollary follows using \cite{wojcik22}. We give an alternate proof using subdifferential calculus.
  \begin{corollary}\label{theorem_com_orth_epsi_kp}
  	Let $A,B\in \M_{m\times n}(\C).$ Let $2\leq p<\infty.$ Let $\varepsilon\in [0,1).$ Then \\$A\perp_{\|\cdot\|_{(p,k)}}^\varepsilon B$ if and only if there exist $k$ orthonormal vectors $v_1,v_2,\ldots,v_k$ satisfying $$A^*Av_i=\sigma_i^2(A)v_i \,\, \text{ for all } 1\leq i \leq k$$ and there exists $z_0\in \C,$ $|z_0|\leq 1$ such that $$\frac{1}{\|A\|_{(p,k)}^{p-1}}\sum_{i=1}^{k}\langle (A^*A)^\frac{p-2}{2}A^*Bv_i,v_i\rangle+z_0\varepsilon\|B\|_{(p,k)}=0.$$
  \end{corollary}
  
  \begin{proof}
  	Consider the linear map $S:\C \to \M_{m\times n}(\C)$ defined by $S(\lambda) = \lambda B$, and the continuous affine map $L:\C \to \M_{m\times n}(\C)$ defined by $L(\lambda) = S(\lambda) + A$. Additionally, consider the continuous convex function $f_3:\M_{m\times n}(\C) \to \mathbb{R}$ defined by $f_3(X) = \|X\|_{(p,k)}^2$, and the function $f_4:\C \to \mathbb{R}^+$ defined by $f_4(\lambda) = 2\varepsilon|\lambda|\|A\|_{(p,k)}\,\|B\|_{(p,k)}$. Since $A\perp_{\|\cdot\|_{(p,k)}}^\varepsilon B$, it follows that $f_3\,\circ L+f_4$ attains its minimum at zero. Then, by Proposition \ref{basic_minma_subd}, Proposition \ref{subdif_affine} and  Proposition \ref{sum}, 
  	
  	we obtain $A\perp_{\|\cdot\|_{(p,k)}}^\varepsilon B$ if and only if

  	\begin{equation}\label{iffsubmin1} 
  		\begin{aligned}
  			0 & \in  \partial\big( f_3\, \circ \,L+f_4 \big)(0)\\
  			&= S^*\partial f_3(A)+\partial f_4(0)\\
  			&= S^*\partial \|\cdot\|_{(p,k)}^2(A)+ \{2\|A\|_{(p,k)}\|B\|_{(p,k)}z:|z|\leq 1 \}\\
  			&= 2 \|A\|_{(p,k)}\, S^*\partial (\|A\|)+\{2\|A\|_{(p,k)}\|B\|_{(p,k)}z:|z|\leq 1 \}\\
  			&= 2 \|A\|_{(p,k)}\, \operatorname{conv}\{S^*(\mathcal{K})\} +\{2\|A\|_{(p,k)}\|B\|_{(p,k)}z:|z|\leq 1 \},\\
  		\end{aligned} 
  	\end{equation} 
  	where $\mathcal{K}$ is the set defined in Remark \ref{remark_ext}. By \cite[Lemma 3.1]{KyFan_ortho_Grover}, we get the convexity of $S^*(\mathcal{K}).$ Thus, by \eqref{iffsubmin1}, we get $$A\perp_{\|\cdot\|_{(p,k)}}^\varepsilon B\text{ if and only if }
  	0\in 2\|A\|_{(p,k)}\{S^*(\mathcal{K})\}+\{2\|A\|_{(p,k)}\|B\|_{(p,k)}z:|z|\leq 1 \}.$$ This implies our required result.
  \end{proof}
  
    \begin{remark}
  	The concept of real $\varepsilon$-orthogonality is also presented in \cite{epsilon_app}. A proof analogous to that of Corollary \ref{theorem_com_orth_epsi_kp} gives the following characterization for the real $\varepsilon$-orthogonality.
  	Let $2\leq p<\infty.$ Let $\varepsilon\in [0,1).$ Then \begin{equation}\label{re_cond_equ}
  		\|A+t B\|_{(p,k)}^2\geq \|A\|_{(p,k)}^2-2\varepsilon\|A\|_{(p,k)}\|tB\|_{(p,k)}\, \text{ for all }t \in \R
  	\end{equation} if and only if there exist $k$ orthonormal vectors $v_1,v_2,\ldots,v_k$ satisfying \begin{equation*}\label{singularvector for A}
  		A^*Av_i=\sigma_i^2(A)v_i \text{ for all } 1\leq i \leq k\end{equation*}
  	and there exists $t_0\in \R,$ $|t_0|\leq 1$ such that 
  	\begin{equation*}
  		\frac{1}{\|A\|_{(p,k)}^{p-1}}\sum_{i=1}^{k}\Re\langle (A^*A)^{p-2}A^*Bv_i,v_i\rangle+t_0\varepsilon\|B\|_{(p,k)}=0.
  	\end{equation*}
  	
  \end{remark}
  
  For $\varepsilon=0$, we get a characterization for Birkhoff-James orthogonality with respect to Ky Fan $p$-$k$ norms as follows.
  \begin{corollary}\label{BirkChara_2}
  	Let $A,B\in \M_{m\times n}(\C).$ Let $2\leq p<\infty$ and $1\leq k\leq n$. Then \begin{equation*}\label{bir_Orth_com}
  		\|A+\lambda B\|_{(p,k)}\geq \|A\|_{(p,k)}\quad \text{ for all }\lambda\in \C
  	\end{equation*} if and only if there exist $k$ orthonormal vectors $v_1,v_2,\ldots,v_k$ satisfying $$A^*Av_i=\sigma_i^2(A)v_i \text{ for all } 1\leq i \leq k$$ such that $$\sum_{i=1}^{k}\langle (A^*A)^\frac{p-2}{2}A^*Bv_i,v_i\rangle=0.$$
  \end{corollary}
  
  We point out that for $p=1$, a characterization of Birkhoff-James orthogonality is given in \cite{KyFan_ortho_Grover} for the Ky Fan $k$ norms. For $k=n$ and $1<p<\infty$, these have been given in \cite{bhatia1999orthogonality}. For $k=n$ and $p=1$, a characterization was given in \cite{lischneider}. 
  
   The concept of norm parallelism is closely related to Birkhoff-James orthogonality.
  A matrix $A$ is said to be norm-parallel to $B$ if there exists $\lambda\in\C,\,|\lambda|=1$ such that 
  $\|A+\lambda B\|_{(p,k)}=\|A\|_{(p,k)}+\|B\|_{(p,k)}$ (see \cite{RankoneSeddik}). It is denoted as $A\parallel B.$ By~\cite[Theorem 2.4]{Zamani2016},
  $A\parallel B$	if and only if $A \perp_B(\| B\|_{(p,k)}A+\alpha \|A\|_{(p,k)} B),$ for some $\alpha \in\C,|\alpha|=1.$ Thus, by Corollary \ref{BirkChara_2}, we get the following characterization.
  
  \begin{corollary}
  	
  	Let $A,B\in\M_{m\times n}(\C).$ Then  $A\parallel B$ if and only if there exist $k$ orthonormal vectors $v_1,v_2,\ldots,v_k$ satisfying $$A^*Av_i=\sigma_i^2(A)v_i \text{ for all } 1\leq i \leq k$$ and a scalar $\lambda\in\C,\,|\lambda|=1$ such that $$\sum_{i=1}^{k}\langle (A^*A)^\frac{p-2}{2}A^*Bv_i,v_i\rangle=\lambda \|A\|_{(p,k)}^{(p-1)}\|B\|_{(p,k)}.$$
  \end{corollary}

 We now turn our attention to orthogonality to a subspace with respect to the Ky Fan $p$-$k$ norms and provide a necessary condition.
  \begin{theorem}\label{subspace_direct1}
  	Let $2\leq p\leq \infty,$ and $1\leq k\leq n_0.$ Let $A \in \mathbb{M}_{m \times n}(\mathbb{C})$. Let $\mathcal{M}$ be a subspace of $\mathbb{M}_{m \times n}(\mathbb{C})$. If $$\|A\|_{(p,k)} \leq \|A + B\|_{(p,k)} \quad \text{for all } B \in \mathcal{M},$$ then there exist density matrices $T_1, \ldots, T_k$ such that, for each $1 \leq i \leq k$, $A^* A T_i = \sigma_i^2(A) T_i,$ and $
  	\frac{1}{\|A\|_{(p,k)}^{p-1}} A (A^* A)^{\frac{p-2}{2}} \sum\limits_{i=1}^k T_i \in \mathcal{M}^\perp.$
  	
  \end{theorem}
  \begin{proof}
  	Let $S_1 : \mathcal{M} \to \mathbb{M}_{m \times n}(\mathbb{C})$ denote the inclusion map, and let $L_1 : \mathcal{M} \to \mathbb{M}_{m \times n}(\mathbb{C})$ be a continuous affine map defined by $L_1(B) = S_1(B) + A$. Let $F : \mathbb{M}_{m \times n}(\mathbb{C}) \to \mathbb{R}$ be a continuous convex defined by $F(X) = \|X\|_{(p,k)}$. Since $\|A\|_{(p,k)} \leq \|A + B\|_{(p,k)}$ for all $B \in \mathcal{M}$, it follows that $F \circ L_1$ attains its minimum at zero. Then, by Propositions \ref{basic_minma_subd} and \ref{subdif_affine}, we get, $$\|A\|_{(p,k)} \leq \|A + B\|_{(p,k)} \, \text{ for all } B \in \mathcal{M}$$ if and only if
  	\begin{equation}\label{iffsubminimal1}
  			0 \in \partial (F \circ L_1)(0) = S_1^* \partial F(A) = S_1^* \partial (\| A\|_{(p,k)}).
  	\end{equation}
  	By Theorem \ref{subdif_kyFan KP}, there exist numbers $\alpha_1,\ldots,\alpha_{n_1}$ with $0\leq \alpha_j\leq 1,$ $\sum\limits_{j=1}^{n_1}\alpha_j=1,$ and for each $1\leq j\leq n_1$, there exist set of orthonormal vectors $\{v_1^{(j)},\ldots,v_k^{(j)}\}$ such that
  	\begin{equation}\label{singular_density1}A^*Av_{i}^{(j)}=\sigma_i^2(A)v_i^{(j)}\quad \text{ for all } i=1,\ldots,k,\text{  and }\end{equation}
  	\begin{equation}\label{S_1^* equation}
  		S_1^*\left(\frac{1}{\|A\|_{(p,k)}^{p-1}}A(A^*A)^\frac{p-2}{2}\sum_{j=1}^{n_1}\alpha_j\left(\sum_{i=1}^{k}v_i^{(j)}{v_i^{(j)}}^*\right)\right)=0.
  	\end{equation}
  	Let $T_i=\sum\limits_{j=1}^{n_1}\alpha_jv_i^{(j)}{v_i^{(j)}}^*.$ Then each $T_i$ is a density matrix and by \eqref{singular_density1}, we get $A^*AT_{i}=\sigma_i^2(A)T_i.$ Note that $S_1^*:\M_{m\times n}(\C)\to\Mc$ is the orthogonal projection onto $\Mc.$ Thus, by \eqref{S_1^* equation} we get 
  	$$\frac{1}{\|A\|_{(p,k)}^{p-1}}A(A^*A)^\frac{p-2}{2}\sum_{i=1}^{k}T_i\in\Mc^\perp.$$ This completes the proof.   \end{proof}
  	
  	Let $m_1\in\N.$ For $i=1,\dots,k,$ let 
  	\begin{equation}\label{subs_conv_eq1}T_i=\sum\limits_{j=1}^{m_1}\alpha_j\, v_i^{(j)}{v_i^{(j)}}^*
  	\end{equation} be a density matrix such that, for each fixed $j$, $\{v_1^{(j)},\dots,v_k^{(j)}\}$ is the set of orthonormal vectors satisfying
  	$$A^*A\,v_i^{(j)}=\sigma_i^2(A)\,v_i^{(j)} ,$$
  	where $\alpha_j\ge0$ and $\sum\limits_{j=1}^{m_1}\alpha_j=1.$ 
   In the following theorem, for the matrices $T_i$ defined in \eqref{subs_conv_eq1}, we obtain the converse of Theorem \ref{subspace_direct1}.
    Let $\|\cdot\|_{(p,k)}^*$ denote the dual norm of the Ky Fan $p$-$k$ norm.
   \begin{theorem}
   	Let $2\leq p\leq \infty,$ and $1\leq k\leq n_0.$ Let $A \in \mathbb{M}_{m \times n}(\mathbb{C})$. Let $\mathcal{M}$ be a subspace of $\mathbb{M}_{m \times n}(\mathbb{C})$. For $i=1,\ldots,k,$ let $T_i$ be the density matrices defined in \eqref{subs_conv_eq1} and $
   	\frac{1}{\|A\|_{(p,k)}^{p-1}} A (A^* A)^{\frac{p-2}{2}} \sum\limits_{i=1}^k T_i \in \mathcal{M}^\perp.$ Then $$\|A\|_{(p,k)} \leq \|A + B\|_{(p,k)} \quad \text{for all } B \in \mathcal{M}.$$
   \end{theorem}
   \begin{proof}
  	For $X\in\M_{m\times n}(\C)$ and $1\leq j\leq n_1$, observe that
  	\begin{align*}
  		\Re\ \tr\left(X^*\frac{1}{\|A\|_{(p,k)}^{p-1}}A(A^*A)^{\frac{p-2}{2}}\sum\limits_{i=1}^kv_i^{(j)}{v_i^{(j)}}^*\right)&=\frac{1}{\|A\|_{(p,k)}^{p-1}} \sum\limits_{i=1}^k\sigma_i^{p-2}(A)\Re\langle Av_i^{(j)},X{v_i^{(j)}}\rangle\\
  		&\leq \frac{1}{\|A\|_{(p,k)}^{p-1}} \sum\limits_{i=1}^k\sigma_i^{p-2}(A)\|Av_i^{(j)}\|\|Xv_i^{(j)}\|\\
  		&=\frac{1}{\|A\|_{(p,k)}^{p-1}} \sum\limits_{i=1}^k\sigma_i^{p-1}(A)\|Xv_i^{(j)}\|\\
  		&\leq \frac{1}{\|A\|_{(p,k)}^{p-1}}\left(\sum\limits_{i=1}^k\left(\sigma_i^{p-1}(A)\right)^\frac{p}{p-1}\right)^\frac{p-1}{p}\left(\sum\limits_{i=1}^k\|Xv_i^{(j)}\|^p\right)^{\frac{1}{p}}\\
  		&=\left(\sum\limits_{i=1}^k\|Xv_i^{(j)}\|^p\right)^{\frac{1}{p}}\\
  		&=\left(\sum_{i=1}^{k}\Big(\tr(v_i^{(j)}{v_i^{(j)}}^*X^*Xv_i^{(j)}{v_i^{(j)}}^*)\Big)^\frac{p}{2}\right)^\frac{1}{p}.
  	\end{align*}
Also, by \cite[Theorem 9.2.10]{bhatia2013matrix}, we get 
  $$\sum_{i=1}^{k}\Big(\tr(v_i^{(j)}{v_i^{(j)}}^*X^*Xv_i^{(j)}{v_i^{(j)}}^*)\Big)^\frac{p}{2}\leq
  \sum_{i=1}^{k}\tr\Big(\big(X^*X\big)^\frac{p}{2}v_i^{(j)}{v_i^{(j)}}^*\Big)\leq \|X\|^p_{(p,k)}.$$
  This implies,  $\left\|\frac{1}{\|A\|_{(p,k)}^{p-1}}A(A^*A)^{\frac{p-2}{2}}\sum\limits_{i=1}^kv_i^{(j)}{v_i^{(j)}}^*\right\|_{(p,k)}^*\leq 1.$
  Further, note that $$\frac{1}{\|A\|_{(p,k)}^{p-1}}A(A^*A)^{\frac{p-2}{2}}\sum\limits_{i=1}^kT_i=\sum\limits_{j=1}^{n_1}\alpha_j\left(\frac{1}{\|A\|_{(p,k)}^{p-1}}A(A^*A)^{\frac{p-2}{2}}\sum\limits_{i=1}^kv_i^{(j)}{v_i^{(j)}}^*\right).$$ Thus, by triangle inequality, we get 
 \begin{equation}\label{con_sub_2} \left\|\frac{1}{\|A\|_{(p,k)}^{p-1}}A(A^*A)^{\frac{p-2}{2}}\sum\limits_{i=1}^kT_i\right\|_{(p,k)}^*\leq 1.
 	\end{equation}
  	 For $B\in\Mc,$ equation \eqref{con_sub_2} implies that
  	 \begin{align*}
  	 \|A+B\|_{(p,k)}&\geq\Re\,\tr\left((A+B)^*\frac{1}{\|A\|_{(p,k)}^{p-1}}A(A^*A)^{\frac{p-2}{2}}\sum\limits_{i=1}^kT_i\right)\\
  	 &=\Re\,\tr\left(A^*\frac{1}{\|A\|_{(p,k)}^{p-1}}A(A^*A)^{\frac{p-2}{2}}\sum\limits_{i=1}^kT_i\right)\\
  	 &=\|A\|_{(p,k)}\\
  	 \end{align*}
  	 Thus, we get the required result.
  	  \end{proof}
  	  
  	\section{Best approximants with respect to the Ky Fan $p$-$k$ norms}\label{sec3}
  In this section, we study various properties of best approximants with respect to the Ky Fan $p$-$k$ norms. The first theorem gives a sufficient condition for best approximations to a one-dimensional subspace $\operatorname{span}\{X\}$ to be unique.
   \begin{theorem}\label{uniq_thrm_onedim1}
  	Let $2\leq p<\infty$ and $1\leq k\leq n_0.$ Let $A,X\in\M_{m \times n}(\C)$ with $\operatorname{rank}(X)>n-k.$ Then there is a unique $\alpha_0 \in \C$ such that
  	\begin{equation}\label{uniq_1}
  		\|A-\alpha_0 X\|_{(p,k)}=\min_{\alpha\in \C}\|A-\alpha X\|_{(p,k)}.
  	\end{equation}
 
  \end{theorem}
  \begin{proof} Assume $\alpha_0$ is not unique. Then there exists $\beta_0\in\C$ with $\beta_0\neq \alpha_0$ such that
  	$\|A-\beta_0 X\|_{(p,k)}=\|A-\alpha_0 X\|_{(p,k)}=\min\limits_{\alpha\in \C}\|A-\alpha X\|_{(p,k)}.$
  	Let $A_{\alpha_0} = A - \alpha_0 X$, and  $\gamma_0 = \beta_0 - \alpha_0$. Then $A - \beta_0 X = A_{\alpha_0} - \gamma_0 X$.
  	Note that $A-\beta_0 X \perp_B \gamma_0 X.$ By Corollary \ref{BirkChara_2}, there exist $k$ orthonormal vectors $v_1,v_2,\ldots,v_k$ satisfying $(A - \beta_0 X)^*(A - \beta_0 X)v_i=\sigma_i^2(A - \beta_0 X)v_i$ for all $1\leq i\leq k,$ such that $$
  		\sum_{i=1}^{k}\left\langle \Big((A - \beta_0 X)^*(A - \beta_0 X)\Big)^\frac{p-2}{2}(A - \beta_0 X)^*(\gamma_0 X)v_i,v_i\right\rangle=0.
  	$$
  	This gives \begin{equation}\label{uniq1dim_KyF_pk_eq_4}
  		\sum_{i=1}^{k}\sigma_i^{p-2}(A - \beta_0 X)\left\langle \gamma_0 X v_i,A_{\alpha_0}v_i\right\rangle=\sum_{i=1}^{k}\sigma_i^{p-2}(A - \beta_0 X)(|\gamma_0|^2\|Xv_i\|^2).
  		\end{equation}
  So 
  
  		\begin{align*}
  			\|A - \beta_0 X\|_{(p,k)}^p 
  			&= \sum_{i=1}^{k} \sigma_i^{p}(A - \beta_0 X) \\
  			&= \sum_{i=1}^{k} \sigma_i^{p-2}(A - \beta_0 X) \langle (A - \beta_0 X)v_i, (A - \beta_0 X)v_i \rangle \\
  			&= \sum_{i=1}^{k} \sigma_i^{p-2}(A - \beta_0 X) \left( \|A_{\alpha_0} v_i\|^2 - 2 \Re \langle A_{\alpha_0} v_i, \gamma_0 X v_i \rangle + |\gamma_0|^2\|Xv_i\|^2 \right).\end{align*} 
  			Using \eqref{uniq1dim_KyF_pk_eq_4} in the above equation, we get
  			$$
  			\|A - \beta_0 X\|_{(p,k)}^p= \sum_{i=1}^{k} \sigma_i^{p-2}(A - \beta_0 X) \left( \|A_{\alpha_0} v_i\|^2 - |\gamma_0|^2\|Xv_i\|^2 \right).$$
  		
  			Since $\gamma_0\neq 0 \text{ and } \operatorname{rank}(X)>n-k,$
  				\begin{equation} \label{unique_3} 
  			\begin{split}
  			\|A - \beta_0 X\|_{(p,k)}^p &<\sum_{i=1}^{k} \sigma_i^{p-2} (A - \beta_0 X) \|A_{\alpha_0} v_i\|^2\\
  			&\leq \Big(\sum_{i=1}^{k} \sigma_i^{p} (A - \beta_0 X)\Big)^{\frac{p-2}{p}}\Big( \sum_{i=1}^{k} \|A_{\alpha_0}v_i\|^p\Big)^\frac{2}{p}\\
  			&= \Big(\sum_{i=1}^{k} \sigma_i^{p} (A - \beta_0 X)\Big)^{\frac{p-2}{p}}\Big(\sum_{i=1}^{k} \big(\tr(v_iv_i^*A_{\alpha_0}^*A_{\alpha_0}v_iv_i^*)\big)^\frac{p}{2}\Big)^\frac{2}{p}.
  		\end{split}
  	\end{equation}
  	From \cite[Theorem 9.2.10]{bhatia2013matrix}, we have 
  	$$\sum_{i=1}^{k}\Big(\tr(v_iv_i^*A_{\alpha_0}^*A_{\alpha_0}v_iv_i^*)\Big)^\frac{p}{2}\leq
  	\sum_{i=1}^{k}\tr\Big(\big(A_{\alpha_0}^*A_{\alpha_0}\big)^\frac{p}{2}v_iv_i^*\Big)\leq \|A_{\alpha_0}\|^p_{(p,k)}.$$
  	So, by \eqref{unique_3}, we get
  	$$\|A - \beta_0 X\|_{(p,k)}^p<\|A - \beta_0 X\|_{(p,k)}^{p-2}\|A_{\alpha_0}\|^2_{(p,k)}=\|A - \beta_0 X\|_{(p,k)}^p.$$
  	This is a contradiction.
  \end{proof}
  
  With the help of the above theorem, we show that one of the conjectures raised in \cite{surveyZ} is not true in general.
  First, we introduce some notation and recall results from \cite{surveyZ} that will be useful in what follows.\\ Let 
  $\underbrace{\rho_1,\ldots,\rho_1}_{s_1},\underbrace{\rho_2,\ldots,\rho_2}_{s_2},\ldots,$ $\underbrace{\rho_l,\ldots,\rho_l}_{s_l}$ be the singular values of $R^{(st)}=A-Y^{(st)}$ and let $s_i$ denote the multiplicity of $\rho_i.$ For $k\in\{1,2,\cdots,l\},$ let $t_k=s_1+s_2+\ldots+s_k.$  
  Note that from the construction of $\Mc_j$ (defined in \eqref{subset_seq}), we have
  $$\Mc_{t_{k-1}+1}=\Mc_{t_{k-1}+2}=\cdots=\Mc_{t_{k}},\quad \text{ where } k=1,2,\dots, l.$$ 

  For $k=l,$ we have $\Mc_{t_l}=\Mc_{n_0}=\{Y^{(st)}\}.$ We recall the following results from \cite{surveyZ}.
  
  \begin{pro}\cite[Corollary 10.3]{surveyZ}\label{Skproperty}
  	Let $X\in \Mc$ and let $R=A-X.$
  	\begin{enumerate}
  		\item If $\sigma_i(R)\neq \rho_1$ for some $1\leq i \leq s_1,$ then $\|R\|_\infty>\rho_1,$ that is, $R\notin\Mc_1.$\\
  		\item If $X=R-A\in \Mc_{t_k}$ and $\sigma_i(R)\neq \rho_{k+1}$ for some   $t_k+1\leq i\leq t_{k+1},$ then $\sigma_{t_k+1}(R)>\rho_{k+1},$ that is, $R\notin\Mc_{t_{k+1}}.$\\    
  	\end{enumerate} 
  \end{pro}
  For $1<p<\infty,$ let $R_p=A-Y_p.$ Let $R^{(st)}=A-Y^{(st)}.$
  Let  $Y_\infty\in \Mc$ be a spectral approximation of $A,$ and let $R_\infty=A-Y_\infty.$ Then, by definition of approximations we have 
  $$\|R_\infty\|_\infty\leq \|R_p\|_\infty\leq \|R_p\|_p\leq \|R_\infty\|_p\leq {n_0}^{1/p}\|R_\infty\|.$$
  This further implies that \begin{equation}\label{boundedness2}\sigma_1(R_p)\to \sigma_1(R_\infty)\text{ as }p\to \infty.
  \end{equation}
  
  \begin{pro}\cite[p. 33]{surveyZ}\label{genereal_conv1}
  	 If  $$\lim_{p\to\infty}\sigma_i(R_p)=\sigma_i(R^{(st)})\text{ for all } i=1,2,\cdots,n_0,$$ then
  	$\lim\limits_{p\to\infty}R_p=R^{(st)}.$
  \end{pro} 
  
  \begin{pro}\cite[Corollary 10.5]{surveyZ}\label{1stsingular_val}
  	 $$\lim_{p\to\infty}\sigma_i(R_p)=\sigma_i(R^{(st)})=\rho_1\quad \text{ for } i=1,2,\cdots s_1.$$ 
  \end{pro}
  
  \begin{remark}\label{uniq_strict_con1}
  	\begin{enumerate}
  		\item By \eqref{boundedness2}, the limit of any convergent subsequence of $Y_p$ is a spectral approximation. If the spectral approximation of $A$ is unique, then it will be $Y^{(st)}.$ So $Y^{(st)}$ is the limit for every convergent subsequence of $Y_p.$ Since $Y_p$ is bounded, we get $Y_p\to Y^{(st)}$ as $p\to\infty.$ 
  		\item If $M=\operatorname{span}\{X_0\}$ with $\operatorname{rank}(X_0)=n,$ then by Theorem~\ref{uniq_thrm_onedim1} the spectral approximation of $A$ is unique. Thus, in this case, \eqref{conjecture} holds.
  		\item If $s_1=n_0,$ then by Proposition \ref{genereal_conv1} and Proposition \ref{1stsingular_val}, we get $\lim\limits_{p\to\infty}R_p=R^{(st)}.$ Thus, \eqref{conjecture} holds true.
  	\end{enumerate}
  \end{remark}

  \begin{theorem}\label{2ndsingularvalue}
  	Let $s_1=1.$ Then for each $i=t_1+1,t_1+2,\ldots,t_2$, we have
  	$$\lim_{p\to\infty}\sigma_i(R_p)=\sigma_i(R^{st})=\rho_2.$$
  \end{theorem}
  \begin{proof}
  	Suppose on the contrary that for some $i_0$ $$\lim\limits_{p\to\infty}\sigma_{i_0}(R_p)\neq\rho_2,$$ where $t_1\leq i_0\leq t_2.$ Let $R_{p_j}$ be a convergent subsequence of $R_p$
  	which converges to $R_0.$ This implies $\sigma_{i_0}(R_0)\neq \sigma_{i_0}(R^{st})=\rho_2.$ By \eqref{boundedness2}, $R_0-A\in\Mc_1$ and by Proposition \ref{Skproperty}, it further implies that $\rho_2<\sigma_{i_0}(R_0)$. Then, by the continuity of singular values, there exists  ${N_0}\in \N$ and $Z_0>0$ such that
  	$$\rho_2<Z_0<\sigma_{i_0}(R_{p_j}) \text{ for all } j>N_0.$$
  	Now we can choose a large $j>N_0$ such that $(n_0-1)\rho_2^{p_j}<Z_0^{p_j }.$ Thus
  	\begin{align*}
  		\sum_{i=1}^{n_0}\sigma_i^{p_j}(R^{st})&\leq\rho_1^{p_j}+(n_0-1)\rho_2^{p_j}\\
  		&\leq \sigma_1^{p_j}(R_{p_j})+(n_0-1)\rho_2^{p_j}\\
  		&<\sigma_1^{p_j}(R_{p_j})+\sigma_{i_0}^{p_j}(R_{p_j})\\
  		&\leq \|R_{p_j}\|_{p_j}^{p_j}.
  	\end{align*}
  	This implies that $Y^{st}$ is a better $c_{p_j}$-approximation than $Y_{p_{j}},$ which contradicts the minimality of $Y_{p_j}.$ 
  \end{proof}		
  \begin{theorem}
  	For the spaces $\M_{2\times n}(\C)$ and $\M_{m\times 2}(\C),$  \eqref{conjecture} holds, that is, $$Y_p\to Y^{(st)} \quad \text{ as } p\to\infty. $$
  \end{theorem}
  	\begin{proof}
  		Consider the singular values of $R^{(st)}=A- Y^{(st)}$ as $\sigma_1(R^{(st)})$ and $\sigma_2(R^{(st)}).$ 
  		\begin{enumerate}
  			\item If $\sigma_1(R^{(st)})=\sigma_2(R^{(st)}),$ then by Remark \ref{uniq_strict_con1}, we get $R_p\to R^{(st)}$ as $p\to\infty.$
  			\item If $\sigma_1(R^{(st)})>\sigma_2(R^{(st)}),$ then by Proposition \ref{1stsingular_val} and Theorem \ref{2ndsingularvalue}, we get $$\lim\limits_{p\to\infty}\sigma_i(R_p)=\sigma_i(R^{(st)}), \text{ for } i=1,2.$$ Thus, by Proposition \ref{genereal_conv1}, we get $R_p\to R^{(st)}$ as $p\to\infty.$\\
  			
  			\end{enumerate}	\end{proof}
 	\begin{remark}
  For $1< p<\infty$ and $1\leq k\leq n,$ let $B^{(p,k)}\in \Mc$ denote a best approximation with respect to the Ky Fan $p$-$k$ norm. Let $R^{(p,k)} = A - B^{(p,k)}$, and let $\W^{(p,k)}$ denote the collection of all such $R^{(p,k)}$. The attempt to prove \eqref{conjecture} in \cite{surveyZ} was by contradiction. To do that, the matrix $R^{(p,k)}_{\min}\in \W^{(p,k)}$ is defined as the minimum element in the lexicographic ordering of the set $\W^{(p,k)}.$ It was also shown that there exists a number $z$ and index $j_0,$ such that for all $j>j_0,$ $z<\sigma_{t_k+1}(R_{p_j})$ over a subsequence $\{R_{p_j}\}$ of $\{R_p\}.$ Further, it was required to find a constant $w$ with $0<w<z$ and a matrix $\widetilde{X}^{(p_j)}_{t_k}\in\Mc$ such that $\widetilde{R}^{(p_j)}_{t_k}=A-\widetilde{X}^{(p_j)}_{t_k}$ satisfies $\|\widetilde{R}^{(p_j)}_{t_k}\|_{(p_j,t_k)}\leq\|R_{p_j}\|_{(p_j,t_k)}$ and $\sigma_{{t_k}+1}(\widetilde{R}^{(p_j)}_{t_k})<w.$ 
In \cite[p.33]{surveyZ}, it was conjectured that
   $\widetilde{R}^{(p_j)}_{t_k}$ could be selected as $R^{(p_j,t_k)}_{\min}.$ We now show that this does not hold true in general. To see this, we consider the following example.\\
 Let $A=\begin{bmatrix}
  	1/2&0&0\\
  	0&2&0\\
  	0&0&0
  \end{bmatrix}$ and $X=\begin{bmatrix}
  0&0&0\\
  0&1&0\\
  0&0&1
  \end{bmatrix}\in\M_{3 \times 3}(\C).$ Let $\Mc=\operatorname{span}\{X\}.$ Then we get $R^{(st)}=\begin{bmatrix}
   1/2&0&0\\
  0&1&0\\
  0&0&1
  \end{bmatrix}.$ It follows from Proposition \ref{1stsingular_val} that \begin{equation}\label{3x3case_eq1}\lim\limits_{p\to\infty}\sigma_i(R_p)=\sigma_i(R^{(st)})=1\,\, \text{ for }i=1,2.
  \end{equation} Suppose $\lim\limits_{p\to\infty}\sigma_3(R_p)\neq\sigma_3(R^{(st)}).$ Since $R_p$ is bounded, there exists a subsequence $R_{p_j}$ of $R_p$ such that $R_{p_j}$ converges to $R_0.$ Thus, by \eqref{3x3case_eq1}, $R_0-A\in\Mc_2$ and Proposition \ref{Skproperty} implies that, $\sigma_3(R^{(st)})<\sigma_3(R_0).$ By the continuity of singular values, there exist $N_1\in\N$ and $Z_1>0$ such that 
  \begin{equation}\label{contra_eqn1}
  	\sigma_3(R^{(st)})<Z_1<\sigma_3(R_{p_j}) \,\,\text{ for all } j>N_1.
  \end{equation}
    Now, to establish the claim for this case, we need to find a constant $w_1$ such that $0<w_1<Z_1,$ and a matrix $\widetilde{R}^{(p_j)}_2=A-\widetilde{B}^{(p_j)}_2$ such that   $\|\widetilde{R}^{(p_j)}_2\|_{(p_j,2)}\leq\|R_{p_j}\|_{(p_j,2)}$ and
  \begin{equation}\label{contra_eqn2} 
  	\sigma_3(\widetilde{R}^{(p_j)}_2)<w_1<Z_1.
  \end{equation} 
  We show that such an $\widetilde{R}^{(p_j)}_2$ cannot be equal to $R^{({p_j},2)}_{\min}.$ Suppose they are equal.
  Note that by Theorem \ref{uniq_thrm_onedim1}, the set $\W_{p_j}^{(2)}$ is a singleton. So, let $\W_{p_j}^{(2)}=\{R^{({p_j},2)}\}$. 
  Then \begin{equation}\label{contra_eq3}
  	\widetilde{R}^{(p_j)}_2=R^{({p_j},2)}_{\min}=R^{({p_j},2)}.
  \end{equation}
  Further, by the definitions of $R^{(p_j,2)}$ and $R_{p_j},$ we have
  \begin{equation}\label{pk_inequality1}
  	\sigma_1^{p_j}(R^{(p_j,2)})+\sigma_2^{p_j}(R^{(p_j,2)})\leq \sigma_1^{p_j}(R_{p_j})+\sigma_2^{p_j}(R_{p_j}).
  \end{equation}  
  and
  \begin{equation}\label{Rp_inequality2}
  	\sigma_1^{p_j}(R_{p_j})+\sigma_2^{p_j}(R_{p_j})+\sigma_3^{p_j}(R_{p_j})\leq \sigma_1^{p_j}(R^{(p_j,2)})+\sigma_2^{p_j}(R^{(p_j,2)})+\sigma_3^{p_j}(R^{(p_j,2)}).
  \end{equation}
  Thus,  \eqref{pk_inequality1} and \eqref{Rp_inequality2} together imply that \begin{equation}\label{contra_eq4}
  	\sigma_3(R_{p_j})\leq \sigma_3(R^{(p_j,2)}).
  \end{equation}
  Now inequalities \eqref{contra_eqn1}, \eqref{contra_eqn2}, \eqref{contra_eq3} and \eqref{contra_eq4} give
  $$\sigma_3(R^{(p_j,2)})<w_1<Z_1<\sigma_3(R_{p_j})\leq\sigma_3(R^{(p_j,2)})$$ which is a contradiction. 
  \end{remark}
  We recall the following proposition, which will be useful for further discussion.  
  \begin{pro}\cite[Corollary 1]{tr}\label{SubdiffbestF1}
  	Let $Y$ be solution of \eqref{kp_approx_def11} with respect to the unitarily invariant norm $|\!|\!|\cdot|\!|\!|$ if and only if there exists a $F\in \Mc^{\perp}$ such that $$F\in\partial|\!|\!|A-Y|\!|\!|.$$
  \end{pro}
 The following theorem gives a characterization about the best approximations in the Ky Fan $p$-$k$ norm in terms of their singular values.
  
  \begin{theorem}\label{kthsingulavaleRpk1}
	Let $2\leq p<\infty,$ and $1\leq k\leq n_0.$ Suppose there is an  $R_0^{(p,k)}\in \W_p^{(k)}$ such that $\sigma_k(R_0^{(p,k)})>\sigma_{k+1}(R_0^{(p,k)}).$
	Then for every $ R^{(p,k)}\in \W_p^{(k)} $ and every $ 1\leq i\leq k$,
	$$ \sigma_i({R_0^{(p,k)}})=\sigma_i({R^{(p,k)}}).$$
\end{theorem}

\begin{proof}
	Let $v_1,\ldots,v_k$ be orthonormal vectors such that $${R_0^{(p,k)}}^*R_0^{(p,k)}v_i=\sigma_i^2(R_0^{(p,k)})v_i\quad\text{ for all } 1 \leq i \leq k.$$ Since $\sigma_k(R_0^{(p,k)})>\sigma_{k+1}(R_0^{(p,k)}),$  we get
	$$ \partial \|R_0^{(p,k)}\|_{(p,k)}=\left\{ \frac{1}{\|R_0^{(p,k)}\|_{(p,k)}^{p-1}}  {R_0^{(p,k)}}\left({R_0^{(p,k)}}^*R_0^{(p,k)}\right)^\frac{p-2}{2}\sum_{i=1}^{k}v_iv_i^*\right\}.$$  Since $A-R_0^{(p,k)}$ is a best approximant of $A,$ by Proposition \ref{SubdiffbestF1}, there exists $F\in \Mc^\perp$ such that $F\in  \partial \|R_0^{(p,k)}\|_{(p,k)}.$ Thus \begin{equation}\label{F1equality} F=\frac{1}{\|R_0^{(p,k)}\|_{(p,k)}^{p-1}}  {R_0^{(p,k)}}\left({R_0^{(p,k)}}^*R_0^{(p,k)}\right)^\frac{p-2}{2}\sum_{i=1}^{k}v_iv_i^*.
	\end{equation} 
	Let $R^{(p,k)}\in \W_p^{(k)}.$ Then
	$F$ is also in $ \partial \|R^{(p,k)}\|_{(p,k)}.$ By Theorem \ref{subdif_kyFan KP}, there exist positive numbers $\lambda_1, \ldots, \lambda_m$ with $\sum_{j=1}^m \lambda_j=1$ and for each $1\leq j\leq m$, there exist $k$  orthonormal vectors $x_1^{(j)},x_2^{(j)},\ldots,x_k^{(j)}$ such that ${R^{(p,k)}}^*R^{(p,k)} x_i^{(j)}=\sigma_i^2(R^{(p,k)}) x_i^{(j)}$ for all $1\leq i\leq k$ and
	$$ F=\sum_{j=1}^{m}\lambda_j\left(\, \frac{1}{\|R^{(p,k)}\|_{(p,k)}^{p-1}}{R^{(p,k)}}\left({R^{(p,k)}}^*R^{(p,k)}\right)^\frac{p-2}{2}\sum_{i=1}^{k}x_i^{(j)}{x_i^{(j)}}^*\right).$$ 
	 For $q=\frac{p}{p-1},$ the conjugate index of $p$, we have $$ \left\|\, \frac{1}{\|R^{(p,k)}\|_{(p,k)}^{p-1}}{R^{(p,k)}}\left({R_1^{(p,k)}}^*R_1^{(p,k)}\right)^\frac{p-2}{2}\sum_{i=1}^{k}x_i^{(j)}{x_i^{(j)}}^*\right\|_q=1=\|F\|_q\, \text{ for all }j=1,\ldots,m_1.$$ For $1< q\leq 2,$  $c_q$ norm is strictly convex. So 
	\begin{equation} \label{F2equality}
		F= \, \frac{1}{\|R_1^{(p,k)}\|_{(p,k)}^{p-1}}{R_1^{(p,k)}}\left({R_1^{(p,k)}}^*R_1^{(p,k)}\right)^\frac{p-2}{2}\sum_{i=1}^{k}x_i^{(j)}{x_i^{(j)}}^* \text{ for all }j=1,\ldots,m_1.
	\end{equation}
	From the equality of $F$ given in \eqref{F1equality} and \eqref{F2equality}, we get the required result.

\end{proof}

\section*{Acknowledgments}
   The authors are grateful to the referee for a careful reading of the manuscript and for the suggestions that improved its presentation. The authors are thankful to K. Zi\k{e}tak for bringing her paper to our attention and for introducing several open problems. We also sincerely appreciate the valuable discussions we had with her. The author P. Grover is supported by a grant CRG/ $2023/000595$ funded by the Anusandhan National Research Foundation (ANRF), India.
	 	\bibliographystyle{abbrv}
	\bibliography{M2Final_ref1}

@article{ACCF,
author = {Altwaijry, N. and  Chmieli\'nski, J. and  Conde, C. and Feki, K.},
title = {Approximate orthogonality and its applications to specific classes of linear operators},
journal = {Bulletin des Sciences Mathématiques},
volume = {202},
year = {2025},
pages = {103645},
issn = {0007-4497},
doi = {https://doi.org/10.1016/j.bulsci.2025.103645},
url = {https://www.sciencedirect.com/science/article/pii/S0007449725000715},
}

@article{wojcik22,
title = {Approximate orthogonality in normed spaces and its applications II},
journal = {Linear Algebra and its Applications},
volume = {632},
pages = {258-267},
year = {2022},
issn = {0024-3795},
doi = {https://doi.org/10.1016/j.laa.2021.09.011},
url = {https://www.sciencedirect.com/science/article/pii/S002437952100344X},
author = {W\'ojcik, P.},}

@book{hiriartREalfundamentals,
title={Fundamentals of Convex Analysis},
author={Hiriart-Urruty, J. B. and Lemar{\'e}chal, C.},
publisher={Springer},
year         = {2002}
}

@book{zalinescuGeneralconvex,
 author      = {Z{\u{a}}linescu, C.},
title        = {Convex Analysis in General Vector Spaces},
publisher    = {World Scientific},
address      = {Singapore},
year         = {2002}
}

@book{bhatia2013matrix,
title       = {Matrix Analysis},
author      = {Bhatia, R.},
publisher   = {Springer-Verlag},
series      = {Graduate Texts in Mathematics},
volume      = {169},
year        = {1997}
}

@incollection{surveyZ,
	AUTHOR = { Zi\k{e}tak, K.},
	TITLE = {From the strict {C}hebyshev approximant of a vector to the
	strict spectral approximant of a matrix},
	BOOKTITLE = {\'{E}tudes op\'{e}ratorielles},
	SERIES = {Banach Center Publ.},
	VOLUME = {112},
	PAGES = {307--346},
	PUBLISHER = {Polish Acad. Sci. Inst. Math., Warsaw},
	YEAR = {2017},
	ISBN = {978-83-86806-36-2},
	MRCLASS = {15A60 (47A58 65F99)},
	MRNUMBER = {3754084},
	MRREVIEWER = {Adolf\ Rhodius},
}

@book{marshall1979inequalities,
AUTHOR = {Marshall, Albert W. and Olkin, Ingram and Arnold, Barry C.},
TITLE = {Inequalities: theory of majorization and its applications},
SERIES = {Springer Series in Statistics},
EDITION = {Second},
PUBLISHER = {Springer, New York},
YEAR = {2011},
PAGES = {xxviii+909},
ISBN = {978-0-387-40087-7},
MRCLASS = {26-02 (05-02 26D15 26D20 60E15)},
MRNUMBER = {2759813},
DOI = {10.1007/978-0-387-68276-1},
URL = {https://doi.org/10.1007/978-0-387-68276-1},
}

@article {SSA,
	AUTHOR = {Zi\k{e}tak, K.},
	TITLE = {Strict approximation of matrices},
	JOURNAL = {SIAM J. Matrix Anal. Appl.},
	FJOURNAL = {SIAM Journal on Matrix Analysis and Applications},
	VOLUME = {16},
	YEAR = {1995},
	NUMBER = {1},
	PAGES = {232--234},
	ISSN = {0895-4798},
	MRCLASS = {41A30 (15A99)},
	MRNUMBER = {1311429},
	MRREVIEWER = {George\ A.\ Anastassiou},
	DOI = {10.1137/S089547989224138X},
	URL = {https://doi.org/10.1137/S089547989224138X},
}

@article{lischneider,
AUTHOR = {Li, Chi-Kwong and Schneider, Hans},
TITLE = {Orthogonality of matrices},
JOURNAL = {Linear Algebra Appl.},
FJOURNAL = {Linear Algebra and its Applications},
VOLUME = {347},
YEAR = {2002},
PAGES = {115--122},
ISSN = {0024-3795,1873-1856},
MRCLASS = {15A60 (46B99)},
MRNUMBER = {1899885},
MRREVIEWER = {Peter\ \v Semrl},
DOI = {10.1016/S0024-3795(01)00530-4},
URL = {https://doi.org/10.1016/S0024-3795(01)00530-4},
}

@article{bhatia1999orthogonality,
	title={Orthogonality of matrices and some distance problems},
	author={Bhatia, R. and Semrl, P.},
	journal={Linear algebra and its applications},
	volume={287},
	number={1-3},
	pages={77--85},
	year={1999},
	publisher={Elsevier}
}

@article {tr,
	AUTHOR = {Zi\k{e}tak, K.},
	TITLE = {On approximation problems with zero-trace matrices},
	JOURNAL = {Linear Algebra Appl.},
	FJOURNAL = {Linear Algebra and its Applications},
	VOLUME = {247},
	YEAR = {1996},
	PAGES = {169--183},
	ISSN = {0024-3795,1873-1856},
	MRCLASS = {15A60 (41A30)},
	MRNUMBER = {1412747},
	MRREVIEWER = {I.\ Gavrea},
	DOI = {10.1016/0024-3795(95)00098-4},
	URL = {https://doi.org/10.1016/0024-3795(95)00098-4},
}

@article{watson1linear9k9major4lp,
	title={Linear best approximation using a class of k-major lp norms},
	author={Watson, G. A.},
	journal={Numerical Algorithms},
	volume={8},
	pages={135--146},
	year={1994},
	publisher={Springer},
}

@article {epsilon_app,
	AUTHOR = {Chmieli\'nski, J.},
	TITLE = {On an {$\epsilon$}-{B}irkhoff orthogonality},
	JOURNAL = {JIPAM. J. Inequal. Pure Appl. Math.},
	FJOURNAL = {JIPAM. Journal of Inequalities in Pure and Applied
	Mathematics},
	VOLUME = {6},
	YEAR = {2005},
	NUMBER = {3},
	PAGES = {Article 79, 7},
	ISSN = {1443-5756},
	MRCLASS = {46B20 (46C50)},
	MRNUMBER = {2164320},
	MRREVIEWER = {Carlos Ben\'{\i}tez},
}

@article {KyFan_ortho_Grover,
	AUTHOR = {Grover, P.},
	TITLE = {Orthogonality of matrices in the {K}y {F}an {$k$}-norms},
	JOURNAL = {Linear Multilinear Algebra},
	FJOURNAL = {Linear and Multilinear Algebra},
	VOLUME = {65},
	YEAR = {2017},
	NUMBER = {3},
	PAGES = {496--509},
	ISSN = {0308-1087},
	MRCLASS = {15A60 (15A18 47A12 47A30)},
	MRNUMBER = {3589614},
	MRREVIEWER = {Brian Simanek},
	DOI = {10.1080/03081087.2016.1193118},
	URL = {https://doi.org/10.1080/03081087.2016.1193118},
}

@article {watson_subdiff_cha,
	AUTHOR = {Watson, G. A.},
	TITLE = {Characterization of the subdifferential of some matrix norms},
	JOURNAL = {Linear Algebra Appl.},
	FJOURNAL = {Linear Algebra and its Applications},
	VOLUME = {170},
	YEAR = {1992},
	PAGES = {33--45},
	ISSN = {0024-3795,1873-1856},
	MRCLASS = {15A60},
	MRNUMBER = {1160950},
	MRREVIEWER = {Jorma\ Kaarlo\ Merikoski},
	DOI = {10.1016/0024-3795(92)90407-2},
	URL = {https://doi.org/10.1016/0024-3795(92)90407-2},
}

@article {Bhatt_Priyanka,
	AUTHOR = {Bhattacharyya, Tirthankar and Grover, Priyanka},
	TITLE = {Characterization of {B}irkhoff-{J}ames orthogonality},
	JOURNAL = {J. Math. Anal. Appl.},
	FJOURNAL = {Journal of Mathematical Analysis and Applications},
	VOLUME = {407},
	YEAR = {2013},
	NUMBER = {2},
	PAGES = {350--358},
	ISSN = {0022-247X},
	MRCLASS = {46B20 (46L05 46L08)},
	MRNUMBER = {3071106},
	MRREVIEWER = {\"{O}mer G\"{o}k},
	DOI = {10.1016/j.jmaa.2013.05.022},
	URL = {https://doi.org/10.1016/j.jmaa.2013.05.022},
}

@article {Birkhoff1935,
	AUTHOR = {Birkhoff, G.},
	TITLE = {Orthogonality in linear metric spaces},
	JOURNAL = {Duke Math. J.},
	FJOURNAL = {Duke Mathematical Journal},
	VOLUME = {1},
	YEAR = {1935},
	NUMBER = {2},
	PAGES = {169--172},
	ISSN = {0012-7094},
	MRCLASS = {DML},
	MRNUMBER = {1545873},
	DOI = {10.1215/S0012-7094-35-00115-6},
	URL = {https://doi.org/10.1215/S0012-7094-35-00115-6},
}

@article {BOTtaziMinimal_com,
	AUTHOR = {Bottazzi, T. and Varela, A.},
	TITLE = {Minimal compact operators, subdifferential of the maximum
	eigenvalue and semi-definite programming},
	JOURNAL = {Linear Algebra Appl.},
	FJOURNAL = {Linear Algebra and its Applications},
	VOLUME = {716},
	YEAR = {2025},
	PAGES = {1--31},
	ISSN = {0024-3795},
	MRCLASS = {47B15 (15A60 47A05 47A12 47A30 51M15)},
	MRNUMBER = {4886465},
	DOI = {10.1016/j.laa.2025.03.017},
	URL = {https://doi.org/10.1016/j.laa.2025.03.017},
}

@article {onapproxSSDrag,
AUTHOR = {Dragomir, Sever Silvestru},
TITLE = {On approximation of continuous linear functionals in normed
linear spaces},
JOURNAL = {An. Univ. Timi\c soara Ser. \c Stiin\c t. Mat.},
FJOURNAL = {Analele Universit\u a\c tii din Timi\c soara. Seria \c Stiin\c
te Matematice},
VOLUME = {29},
YEAR = {1991},
NUMBER = {1},
PAGES = {51--58},
MRCLASS = {46B20 (46C50)},
MRNUMBER = {1336199},
}

@book {PriGro_thesis,
	AUTHOR = {Grover, Priyanka},
	TITLE = {Some {P}roblems in {D}ifferential and {S}ubdifferential
	{C}alculus of {M}atrices},
	NOTE = {Thesis (Ph.D.)--Indian Statistical Institute - Kolkata},
	PUBLISHER = {ProQuest LLC, Ann Arbor, MI},
	YEAR = {2014},
	PAGES = {119},
	ISBN = {979-8496-58155-4},
	MRCLASS = {Thesis},
	MRNUMBER = {4380514},
	URL =
	{http://gateway.proquest.com/openurl?url_ver=Z39.88-2004&rft_val_fmt=info:ofi/fmt:kev:mtx:dissertation&res_dat=xri:pqm&rft_dat=xri:pqdiss:28843053},
}

@article {PGmatrixsubspace,
	AUTHOR = {Grover, Priyanka},
	TITLE = {Orthogonality to matrix subspaces, and a distance formula},
	JOURNAL = {Linear Algebra Appl.},
	FJOURNAL = {Linear Algebra and its Applications},
	VOLUME = {445},
	YEAR = {2014},
	PAGES = {280--288},
	ISSN = {0024-3795},
	MRCLASS = {15A60 (46L05 47A06)},
	MRNUMBER = {3151274},
	MRREVIEWER = {Mohammad Sal Moslehian},
	DOI = {10.1016/j.laa.2013.11.040},
	URL = {https://doi.org/10.1016/j.laa.2013.11.040},
}

@article {RCJames_ortho,
	AUTHOR = {James, R. C.},
	TITLE = {Orthogonality and linear functionals in normed linear spaces},
	JOURNAL = {Trans. Amer. Math. Soc.},
	FJOURNAL = {Transactions of the American Mathematical Society},
	VOLUME = {61},
	YEAR = {1947},
	PAGES = {265--292},
	ISSN = {0002-9947},
	MRCLASS = {46.0X},
	MRNUMBER = {21241},
	MRREVIEWER = {R. S. Phillips},
	DOI = {10.2307/1990220},
	URL = {https://doi.org/10.2307/1990220},
}

@article {ArDeKpOnsomegeometric,
	AUTHOR = {Mal, Arpita and Sain, Debmalya and Paul, Kallol},
	TITLE = {On some geometric properties of operator spaces},
	JOURNAL = {Banach J. Math. Anal.},
	FJOURNAL = {Banach Journal of Mathematical Analysis},
	VOLUME = {13},
	YEAR = {2019},
	NUMBER = {1},
	PAGES = {174--191},
	ISSN = {2662-2033},
	MRCLASS = {46B20 (47L05)},
	MRNUMBER = {3892339},
	MRREVIEWER = {Damian Marcin Kubiak},
	DOI = {10.1215/17358787-2018-0021},
	URL = {https://doi.org/10.1215/17358787-2018-0021},
}

@article{Watson1993KyFank,
	AUTHOR = {Watson, G. A.},
	TITLE = {On matrix approximation problems with {K}y {F}an {$k$} norms},
	JOURNAL = {Numer. Algorithms},
	FJOURNAL = {Numerical Algorithms},
	VOLUME = {5},
	YEAR = {1993},
	PAGES = {263--272},
	ISSN = {1017-1398,1572-9265},
	MRCLASS = {65D15 (15A60)},
	MRNUMBER = {1258599},
	DOI = {10.1007/BF02210386},
	URL = {https://doi.org/10.1007/BF02210386},
}

@article {RankoneSeddik,
	AUTHOR = {Seddik, Ameur},
	TITLE = {Rank one operators and norm of elementary operators},
	JOURNAL = {Linear Algebra Appl.},
	FJOURNAL = {Linear Algebra and its Applications},
	VOLUME = {424},
	YEAR = {2007},
	NUMBER = {1},
	PAGES = {177--183},
	ISSN = {0024-3795},
	MRCLASS = {47B47 (47A30)},
	MRNUMBER = {2324383},
	MRREVIEWER = {Borut Zalar},
	DOI = {10.1016/j.laa.2006.10.003},
	URL = {https://doi.org/10.1016/j.laa.2006.10.003},
}

@article {Zietaksubdual1993,
	AUTHOR = {Zi\k{e}tak, K.},
	TITLE = {Subdifferentials, faces, and dual matrices},
	JOURNAL = {Linear Algebra Appl.},
	FJOURNAL = {Linear Algebra and its Applications},
	VOLUME = {185},
	YEAR = {1993},
	PAGES = {125--141},
	ISSN = {0024-3795,1873-1856},
	MRCLASS = {15A60},
	MRNUMBER = {1213175},
	MRREVIEWER = {Roy\ Mathias},
	DOI = {10.1016/0024-3795(93)90209-7},
	URL = {https://doi.org/10.1016/0024-3795(93)90209-7},
}

@article {Prop_spectrZietak1993,
	AUTHOR = {Zi\k{e}tak, K.},
	TITLE = {Properties of linear approximations of matrices in the
	spectral norm},
	JOURNAL = {Linear Algebra Appl.},
	FJOURNAL = {Linear Algebra and its Applications},
	VOLUME = {183},
	YEAR = {1993},
	PAGES = {41--60},
	ISSN = {0024-3795,1873-1856},
	MRCLASS = {15A60},
	MRNUMBER = {1208196},
	MRREVIEWER = {Makoto\ Takaguchi},
	DOI = {10.1016/0024-3795(93)90423-L},
	URL = {https://doi.org/10.1016/0024-3795(93)90423-L},
}

@article {Liesen_Tichy2009,
	AUTHOR = {Liesen, J. and Tich\'y, P.},
	TITLE = {On best approximations of polynomials in matrices in the
	matrix 2-norm},
	JOURNAL = {SIAM J. Matrix Anal. Appl.},
	FJOURNAL = {SIAM Journal on Matrix Analysis and Applications},
	VOLUME = {31},
	YEAR = {2009},
	NUMBER = {2},
	PAGES = {853--863},
	ISSN = {0895-4798,1095-7162},
	MRCLASS = {41A52 (15A60 41A50 65F35)},
	MRNUMBER = {2530280},
	MRREVIEWER = {Fabio\ Di Benedetto},
	DOI = {10.1137/080728299},
	URL = {https://doi.org/10.1137/080728299},
}

@article{zietakextremal,
	author = {Zi\k{e}tak, K.},
	title = {On the characterization of the extremal points of the unit sphere of matrices},
	journal = {Linear Algebra Appl.},
	volume = {106},
	year = {1988},
	pages = {57--75},
	issn = {0024-3795},
	mrclass = {15A60},
	mrnumber = {951827},
	mrreviewer = {Shao Kuan Li},
	doi = {10.1016/0024-3795(88)90023-7},
	url = {https://doi.org/10.1016/0024-3795(88)90023-7}
}

@article {Legg_wardtrace_1985,
	AUTHOR = {Legg, D. A. and Ward, J. D.},
	TITLE = {A canonical trace class approximant},
	JOURNAL = {Proc. Amer. Math. Soc.},
	FJOURNAL = {Proceedings of the American Mathematical Society},
	VOLUME = {93},
	YEAR = {1985},
	NUMBER = {4},
	PAGES = {653--656},
	ISSN = {0002-9939,1088-6826},
	MRCLASS = {47B10 (47A30)},
	MRNUMBER = {776197},
	MRREVIEWER = {Lawrence\ R.\ Williams},
	DOI = {10.2307/2045539},
	URL = {https://doi.org/10.2307/2045539},
}

@book{horn2012matrix,
	title={Matrix analysis},
	author={Horn, Roger A. and Johnson, Charles R.},
	publisher={Cambridge University Press},
	 year    = {1985}
}

@book{bhatia2009positive,
title       = {Matrix Analysis},
author      = {Bhatia, R.},
publisher   = {Princeton University Press},
series      = {Texts and Readings in Mathematics},
volume      = {44},
year        = {2007}
}

@book{singer2013best,
	author    = {I. Singer},
	title     = {Best Approximation in Normed Linear Spaces by Elements of Linear Subspaces},
	publisher = {Springer-Verlag},
	address   = {Berlin},
	year      = {1970}
}

@article {Zamani2016,
	AUTHOR = {Zamani, Ali and Moslehian, Mohammad Sal},
	TITLE = {Norm-parallelism in the geometry of {H}ilbert {$C^*$}-modules},
	JOURNAL = {Indag. Math. (N.S.)},
	FJOURNAL = {Koninklijke Nederlandse Akademie van Wetenschappen.
	Indagationes Mathematicae. New Series},
	VOLUME = {27},
	YEAR = {2016},
	NUMBER = {1},
	PAGES = {266--281},
	ISSN = {0019-3577,1872-6100},
	MRCLASS = {46L08},
	MRNUMBER = {3437749},
	MRREVIEWER = {Daniele\ Puglisi},
	DOI = {10.1016/j.indag.2015.10.008},
	URL = {https://doi.org/10.1016/j.indag.2015.10.008},
}

@article{ying,
	AUTHOR = {Zhang, Ying and Jiang, Lining and Han, Yongheng},
	TITLE = {Constructions of minimal {H}ermitian matrices related to a {$\mathrm{C}^*$}-subalgebra of {$M_n(\mathbb{C})$}},
	JOURNAL = {Proc. Amer. Math. Soc.},
	FJOURNAL = {Proceedings of the American Mathematical Society},
	VOLUME = {151},
	YEAR = {2023},
	NUMBER = {1},
	PAGES = {73--84},
	ISSN = {0002-9939,1088-6826},
	MRCLASS = {15A60 (15B57 46N10)},
	MRNUMBER = {4504608},
	MRREVIEWER = {Robert\ S.\ Doran},
	DOI = {10.1090/proc/16130},
	URL = {https://doi.org/10.1090/proc/16130},
}

@article {andruchow,
	AUTHOR = {Andruchow, Esteban and Mata-Lorenzo, Luis E. and Mendoza,
	Alberto and Recht, L\'azaro and Varela, Alejandro},
	TITLE = {Minimal matrices and the corresponding minimal curves on flag
	manifolds in low dimension},
	JOURNAL = {Linear Algebra Appl.},
	FJOURNAL = {Linear Algebra and its Applications},
	VOLUME = {430},
	YEAR = {2009},
	NUMBER = {8-9},
	PAGES = {1906--1928},
	ISSN = {0024-3795,1873-1856},
	MRCLASS = {46L51 (58B10 58B25)},
	MRNUMBER = {2503942},
	MRREVIEWER = {Daniel\ Belti\c t\u a},
	DOI = {10.1016/j.laa.2008.10.023},
	URL = {https://doi.org/10.1016/j.laa.2008.10.023},
}
\end{document}